\documentclass{amsart}
\usepackage{latexsym,amssymb,amsmath,epsfig}
\usepackage{epsfig,graphicx,latexsym}
\usepackage{pstricks,pstricks-add,pst-math,pst-xkey}
\newtheorem{lemma}{Lemma}[section]
\def\qed{\hfill$\Box$ \bigskip}
\newtheorem{theorem}{Theorem}[section]

\newtheorem{corollary}{Corollary}[section]
\newtheorem{remark}[theorem]{Remark}
\newtheorem{example}[theorem]{Example}

\begin{document}

\title[Langford sequences and a product of digraphs]{Langford sequences and a product of digraphs}

\author{S. C. L\'opez}
\address{%
Departament de Matem\`{a}tiques\\
Universitat Polit\`{e}cnica de Catalunya\\
C/Esteve Terrades 5\\
08860 Castelldefels, Spain}
\email{susana.clara.lopez@upc.edu}

\author{F. A. Muntaner-Batle}
\address{Graph Theory and Applications Research Group \\
 School of Electrical Engineering and Computer Science\\
Faculty of Engineering and Built Environment\\
The University of Newcastle\\
NSW 2308
Australia}
\email{famb1es@yahoo.es}
\date{\today}
\maketitle

\begin{abstract}
Skolem and Langford sequences and their many generalizations have applications in numerous areas. The $\otimes_h$-product is a generalization of the direct product of digraphs. In this paper we use the $\otimes_h$-product and super edge-magic digraphs to construct an exponential number of Langford sequences with certain order and defect. We also apply this procedure to extended Skolem sequences.

\end{abstract}

\section{Introduction}
For $m\le n$, we denote the set $\{m,m+1,\ldots, n\}$ by $[m,n]$.
A {\it Skolem sequence} \cite{NicMar67,Sko57} of order $m$ is a sequence of $2m$ numbers $(s_1,s_2,\ldots,s_{2m})$ such that (i) for every $k\in [1,m]$ there exist exactly two subscripts $i,j\in [1,2m]$ with $s_i=s_j=k$, (ii) the subscripts $i$ and $j$ satisfy the condition $|i-j|=k$.
A Skolem sequence of order $4$ is for instance $(4,2,3,2,4,3,1,1)$.
It is well known that Skolem sequences of order $m$ exist if and only if $m\equiv 0$ or $1$ (mod $4$).

Skolem introduced in \cite{Sko58} what is now called a {\it hooked Skolem sequence} of order $m$, where there exists a zero at the second to last position of the sequence containing $2m+1$ elements. Later on, in 1981, Abrham and Kotzig \cite{AbrKot81}
introduced the \textit{extended Skolem sequence}, where the zero is allowed to appear in any position of the sequence. Notice that from every Skolem sequence we can obtain two trivial extended Skolem sequences just by adding a zero either in the first or in the last position. In this paper, all extended Skolem sequences that we refer to are non trivial, unless otherwise specified.

Let $d$ be a positive integer. A {\it Langford sequence} of order $m$ and defect $d$ \cite{Sim} is a sequence $(l_1,l_2,\ldots, l_{2m})$ of $2m$ numbers such that (i) for every $k\in [d,d+m-1]$ there exist exactly two subscripts $i,j\in [1,2m]$ with $l_i=l_j=k$, (ii) the subscripts $i$ and $j$ satisfy the condition $|i-j|=k$. Langford sequences, for $d=2$, where introduced in \cite{Langford} and they are referred as \textit{perfect Langford sequences}. Notice that, a Langford sequence of order $m$ and defect $d=1$ is a Skolem sequence of order $m$.

Bermond, Brower and Germa on one side \cite{BerBroGer} and Simpson on the other side \cite{Sim} showed that Langford sequences of order $m$ and defect $d$ exist if and only if the following conditions hold:
(i) $m\ge 2d-1$, and
(ii) $m\equiv 0$ or $1$ (mod $4$) if $d$ is odd; $m\equiv 0$ or $3$ (mod $4$) if $d$ is even.



Denote by $\sigma_m$ the number of Skolem sequences of order $m$. It is clear that if $m\equiv 2$ or $3$ (mod $4$) then $\sigma_m=0$. Abraham showed in \cite{Abr} the next result.

\begin{theorem}\cite{Abr}\label{theo_lower_bound_Skolem}
It is $\sigma_m\ge 2^{\lfloor m/3\rfloor}$ for every $m\equiv 0$ or $1$ (mod $4$).
\end{theorem}

For the graph theory notation and terminology used in this paper, unless otherwise specified, we follow \cite{BaMi,CH,G,Wa,W}. However, in order to make this paper reasonably self contained, we mention that by a $(p,q)$-graph we mean a graph of order $p$ and size $q$. We also point out that we allow graphs (and digraphs) to have loops. If we need to consider graphs without loops nor multiple edges we will refer to them as simple graphs.
The underlying graph of a digraph $D$, und$(D)$, is the graph obtained from $D$ after removing the orientation of the arcs.
In general, we say that a digraph $D$ admits a labeling $f$ if its underlying graph admits the labeling $f$.
Bloom and
Ruiz \cite{BloRui96} introduced a generalization of {\it graceful
labelings} (see \cite{G} for a formal definition of graceful
labeling), that they called {\it $k$-equitable labelings}.
Let $G$ be a $(p,q)$-graph and let $g:V(G) \longrightarrow
\mathbb{Z}$ be an injective function with the property that the new
function $h:E(G) \longrightarrow \mathbb{N}$ defined by the rule
$h(uv)=|g(u)-g(v)|$ for every $ uv\in E$ assigns the same integer
to exactly $k$ edges. Then $g$ is said to be a $k$-equitable labeling and $G$
a $k${\it -equitable graph}.
In  \cite{BloRui96} the  authors called a $k$-equitable
labeling, {\it optimal}, when $g$ assigns all the elements from the set
$[1,p]$ to the elements of $V(G)$. Barrientos \cite{Ba2} called a $k$-equitable labeling {\it complete} if the induced edge labels are all the elements in $[1,w]$, where $w$ is the number of distinct edge-labels. Although, $1$-equitable labelings are defined in the context of simple graphs, it is not hard to extend the concept to graphs with loops, where the label of any loop is zero.

At this point, for any Skolem sequence of order $m$, that is to say, of $2m$ elements, we define a {\it Skolem labeling} $g$ of the directed matching $m\overrightarrow{K_2}$, up to isomorphism. Let $f:E(m\overrightarrow{K_2})\rightarrow [1,m]$ be any bijective function such that if $f(u,v)=k$, then the Skolem labeling $g$ of $m\overrightarrow{K_2}$ assigns to $u$ one of the two positions occupied by $k$ and to $v$ the other position occupied by $k$, in such a way that the label of $u$ is strictly smaller than the label of $v$. See the following example.

\begin{example}\label{ex: Matching_of_Skolem_4}
Consider the Skolem sequence $(4,2,3,2,4,3,1,1)$. Then, the corresponding Skolem labeling of the matching $4\overrightarrow{K_2}$ is shown in Fig. \ref{Fig_2}.

\begin{figure}[h]
  \centering
  \includegraphics[width=97pt]{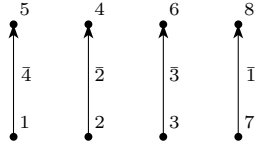}\\
  \caption{The Skolem labeling of $4\protect\overrightarrow{K_2}$ associated to  $(4,2,3,2,4,3,1,1)$.}\label{Fig_2}
\end{figure}

\end{example}

Notice that, every Skolem labeling $g$ of $m\overrightarrow{K_2}$ is a bijective function $g:V(m\overrightarrow{K_2})\rightarrow [1,2m]$ such that, the set of differences $\{g(v)-g(u): (u,v)\in E(m\overrightarrow{K_2})\}=[1,m]$. Thus, $g$ is a complete optimal $1$-equitable labeling of $m\overrightarrow{K_2}$. Moreover, any complete optimal $1$-equitable labeling of $m\overrightarrow{K_2}$ raises an associated Skolem sequence.

In a similar way, we can define a {\it Langford labeling} of $m\overrightarrow{K_2}$, up to isomorphism, for any Langford sequence of order $m$ and defect $d$. Let $f:E(m\overrightarrow{K_2})\rightarrow [d,d+m-1]$ be a bijective function defined as follows: if $f(u,v)=k$, then the Langford labeling $g$ of $m\overrightarrow{K_2}$ assigns to $u$ one of the two positions occupied by $k$ and to $v$ the other position occupied by $k$, in such a way that the label of $u$ is strictly smaller than the label of $v$.
Once again, notice that every Langford labeling $g$ of $m\overrightarrow{K_2}$ is a bijective function $g:V(m\overrightarrow{K_2})\rightarrow [1,2m]$ and, the set of differences $\{g(v)-g(u): (u,v)\in E(m\overrightarrow{K_2})\}=[d,d+m-1]$. Thus, $g$ is an optimal $1$-equitable labeling of $m\overrightarrow{K_2}$. Moreover, any optimal $1$-equitable labeling of $m\overrightarrow{K_2}$, with consecutive set of induced differences on the edges, raises an associated Langford sequence of defect $d$, where $d$ is the minimum of this set.

The bijection among Skolem and Langford sequences and the corresponding Skolem and Langford labelings of $m\overrightarrow{K_2}$ can also be generalized when dealing with hooked (or extended) Skolem sequences, in a natural way. In this case, instead of considering a labeling of $m\overrightarrow{K_2}$, we consider a labeling of $m\overrightarrow{K_2}\cup \overrightarrow{L}$, where the vertex of the loop is labeled with the position of zero in the given sequence. We will refer to this labeling as an \textit{extended Skolem labeling of} $m\overrightarrow{K_2}\cup \overrightarrow{L}$.

Skolem, Langford sequences and their many generalizations have applications in numerous areas, see for instance \cite{FranMen2009}. Although their origin is in the fifties, many recent papers have been contributed to their study and applications from different points of view, as for instance \cite{MatNorSha,MorLin,ShaSil12,ShaSil14}.
In this paper, we study Skolem and Langford sequences through (extended) Skolem and Langford labelings of $m\overrightarrow{K_2}$. By multiplying a Skolem labeled matching $m\overrightarrow{K_2}$ by a particular family of labeled $1$-regular digraphs of order $n$, we obtain a Langford labeled matching $(mn)\overrightarrow{K_2}$. We will show that this procedure can also be applied to obtain Langford sequences from existing ones, and that in this way, we can obtain a lower bound for the number of Langford sequences, for particular values of the defect. We also extend this procedure to hooked and extended Skolem sequences. The organization of the paper is the following one. Section \ref{section: the tools} contains the necessary terminology and previous known results. Section \ref{section_Skolem_Langford} is focused on the study of Skolem and Langford sequences, and finally, Section \ref{section_Hooked_Skolem_Langford} shows an extension to hooked and extended Skolem sequences.

\section{The tools: labelings and the $\otimes_h$-product}
\label{section: the tools}

We start this section by completing the terminology about labelings that we use in the paper.
Kotzig and Rosa defined in 1970 \cite{KotRos70} the concept of edge-magic graphs and edge-magic labelings for simple graphs as follows. Let $G$ be a simple $(p,q)$-graph. A bijective function $f:V(G)\cup E(G)\rightarrow [1,p+q]$ is an {\it edge-magic labeling} of $G$ if the sum $f(u)+f(uv)+f(v)=k,$ for every $uv\in E(G)$. If such a function exits then $G$ is called an {\it edge-magic graph}. The constant $k$ is the {\it valence} of the labeling in \cite{KotRos70}. However, different authors have denoted the valence by other names, as for instance, the {\it magic sum} \cite{Wa} or the {\it magic weight} \cite{BaMi}.


Enomoto et al. in \cite{E} defined in 1998 the concepts of super edge-magic labelings and of super edge-magic graphs as follows. A {\it super edge-magic labeling} of a simple $(p,q)$-graph $G$ is an edge-magic labeling $f$ that has the extra property that $f(V(G))=[1,p]$. In that case, $G$ is called a {\it super edge-magic graph}. However, Acharya and Hegde had introduced in \cite{AH} an equivalent concept under the name {\it strongly indexable graphs}.
In this paper, we will use a more general definition of super edge-magic labelings that does not restrict only to simple graphs, but to graphs that admit at most one loop attached at each vertex in a natural way. That is, the magic sum of a loop is obtained  by adding the label on the loop plus twice the label of the vertex of the loop. Such a generalization was implicitly provided by Figueroa et al. in \cite{F1}.
One of the key ideas when dealing with super edge-magic labelings is that to obtain such a labeling of a graph it is enough to exhibit the labels of the vertices.

\begin{lemma}\cite{F2}\label{lemma_SEM_consecutives_sums}
A $(p,q)$-graph $G$ is super edge-magic if and only if there exists a bijective function $f:V(G)\rightarrow [1,p]$ such that the set $S=\{f(u)+f(v):\ uv\in E(G)\}$ consists of $q$ consecutive integers. In such case, $f$ extends to a super edge-magic labeling of $G$ with magic sum $p+q+s$, where $s=\min (S)$.
\end{lemma}

Unless otherwise specified, whenever we refer to a function as a super edge-magic labeling we will assume that it is a function $f$ as in Lemma \ref{lemma_SEM_consecutives_sums}.
In \cite{F1} Figueroa et al., introduced the concept of {\it super
 edge-magic digraph} as follows:
a digraph $D=(V,E)$ is super edge-magic if its underlying graph is
super edge-magic.

\subsection{The $\otimes_h$-product} The $\otimes_h$-product was introduced by Figueroa et al. \cite{F1} as a tool to obtain (super) edge-magic graphs from existing ones. Later on, this product has been applied to different types of labelings. See for instance \cite{ILMR,LopMunRiu1,LopMunRiu8}.
Let $D$ be a digraph and let $\Gamma $ be a family of digraphs such that $V(F)=V$, for every $F\in \Gamma$. Consider any function $h:E(D)\longrightarrow\Gamma $.
Then the product $D\otimes_{h} \Gamma$ is the digraph with vertex set the Cartesian product $V(D)\times V$ and $((a,x),(b,y))\in E(D\otimes_{h
    }\Gamma)$ \index{$\otimes_h$-product} if and only if $(a,b)\in E(D)$ and $(x,y)\in
    E(h ((a,b)))$. Let $A(D)$ and $A(F)$ be the adjacency matrices of $D$ and $F\in \Gamma$, respectively, when the vertices of $D$ are indexed as $V(D)=\{a_1,a_2,\ldots, a_m\}$ and the vertices of $F$ as $V=\{x_1,x_2,\ldots, x_n\}$. Let the vertices of $D\otimes_{h} \Gamma$ be indexed as
$\{(a_1,x_1),\ldots, (a_1,x_n),(a_2,x_1),\ldots, (a_m,x_n)\}.$ Then,
the adjacency matrix of $D\otimes_{h} \Gamma$, $A(D\otimes_{h} \Gamma)$, is obtained by multiplying every $0$ entry of $A(D)$ by the $|V|\times |V|$ null matrix and every $1$ entry of $A(D)$ by $A(h(a,b))$, where $(a,b)$ is the arc related to the corresponding $1$ entry. Notice that when $h$ is constant, the adjacency matrix of $D\otimes_{h} \Gamma$ is just the classical Kronecker product $A(D)\otimes A(h(a,b))$. Thus, when $|\Gamma |=1$, we just write $D\otimes\Gamma$.

\begin{example}\label{Example_1: product of a matching by rotated}
Let $D$ be the digraph defined by $V(D)=[1,8]$ and $E(D)=\{(1,5),(2,4),(3,6),(7,8)\}$. Let $\Gamma=\{F_1,F_2\}$, where $V(F_1)=V(F_2)=[1,3]$, $E(F_1)=\{(1,1), (2,3),(3,2)\}$ and $E(F_2)=\{(1,2),(2,1),(3,3)\}$. Consider the function $h:E(D)\rightarrow \Gamma$ defined by $h((1,5))=h((2,4))=F_1$ and $h((3,6))=h((7,9))=F_2$. Then, the digraph $D\otimes_h\Gamma$ is shown in Fig. \ref{Fig_1}.
\begin{figure}[h]
  \centering
  \includegraphics[width=337pt]{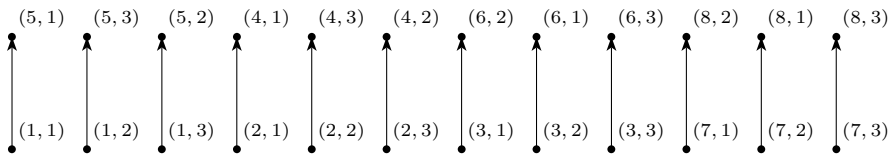}\\
  \caption{The digraph $D\otimes_h\Gamma$. }\label{Fig_1}
\end{figure}
\end{example}

\subsection{The $\otimes_h$-product applied to labelings}

Let $\mathcal{S}_n^k$ denote the set of all super edge-magic labeled digraphs of order and size equal to $n$ with the minimum sum of the labels of the adjacent vertices being $k$ (that is, the minimum of the set $S$ introduced in Lemma \ref{lemma_SEM_consecutives_sums}), where each vertex takes the name of its label. L\'opez et al. obtained in \cite{LopMunRiu8} the following result.

\begin{theorem}\cite{LopMunRiu8}\label{producte_super_k}
Assume that
$D$ is any (super) edge-magic digraph and
  $h$ is any function $h:E(D)\rightarrow \mathcal{S}_n^k$.
Then und$(D\otimes_h \mathcal{S}_n^k)$ is (super) edge-magic.
\end{theorem}

Let $\mathcal{S}_n$ denote the set of all $1$-regular super edge-magic labeled digraphs of order $n$. In \cite{BaLinMunRiu}, Ba\v{c}a et al. gave a lower bound for the size of $\mathcal{S}_n$.

\begin{theorem}\cite{BaLinMunRiu}\label{theo:numberSEMofCICLES}
Let $C_n$ be a cycle on $n$ vertices, $n\ge 11$ odd. The number of super edge-magic labelings of the cycle $C_n$ is at least $5/4 \cdot 2^{\lfloor (n-1)/3\rfloor}+1.$
\end{theorem}

Let $\Sigma_n$ be the set of all $1$-regular digraphs of order $n$. Figueroa et al. obtained in \cite{F1} the next result.
\begin{theorem}\cite{F1}\label{theo_copies_of_forest}
Let $F$ be an acyclic graph.
Consider any function $h:E(\overrightarrow{F})\rightarrow \Sigma_n$. Then, $\overrightarrow{F}\otimes_h\Sigma_n\cong n\overrightarrow{F}$.
\end{theorem}

A {\it rotation super edge-magic digraph} of order $n$ was introduced in \cite{LopMunRiu1}. Let $M=(a_{i,j})$ be a square matrix of order $n$. The matrix $(a^R_{i,j})$ is the {\it rotation of the matrix} $M$, denoted by $M^R$, when $a^R_{i,j}=a_{n+1-j,i}.$ Graphically this corresponds to a rotation of the matrix by $\pi/2$ radiants clockwise. A digraph $S$ is said to be a {\it rotation super edge-magic {\it digraph} of order} $n$, if its adjacency matrix is the rotation of the adjacency matrix of an element in $\mathcal{S}_n$. The expression $\mathcal{RS}_n$ denotes the set of all digraphs that are rotation super edge-magic digraphs of order $n$.

The following lemma and application of the $\otimes_h$-product to $k$-equitable digraphs were given in \cite{LopMunRiu1}.
\begin{lemma}\cite{LopMunRiu1}\label{coro: the unicity of $i-j=k$}
Let $S$ be a digraph in $\mathcal{RS}_n$ and let $k$ be any integer.
 If $|k|\le (n-1)/2$ then there exists an unique arc $(i,j)\in E(S)$ such that $i-j=k$.
\end{lemma}

\begin{theorem}\cite{LopMunRiu1}\label{theo: product_k_equitable}
Let $D$ be an (optimal) $k$-equitable digraph and let $h:E(D)\rightarrow \mathcal{RS}_n$ be any function. Then $D\otimes_h\mathcal{RS}_n$ is (optimal) $k$-equitable.
\end{theorem}

\begin{remark}\label{remark: induced labeling product}
As a key point in what follows, we want to mention how the $k$-equitable labeling of Teorem \ref{theo: product_k_equitable} is obtained from the labelings of the elements involved in it. If we assume that each vertex of $D$ is identified with the label assigned to it by a $k$-equitable labeling then, a vertex $(a,i)$ of $D\otimes_h\mathcal{RS}_n$ is labeled with $n(a-1)+i$.
\end{remark}

\section{Using SEM labelings of $(p,p)$-graphs and Skolem sequences to generate Langford sequences}
\label{section_Skolem_Langford}

The problem of counting Langford sequences has proven to be an interesting and challenging one. This section is devoted to introduce new techniques that will allow us to find lower bounds for the number of Langford sequences of certain order and defect $d$ for some particular values of $d$. We feel that the techniques introduced are unexpected since we use graph labelings and the $\otimes_h$-product. By using this procedure, we have a different way to understand the problem.

In the next theorem, we show that for every Langford sequence $L$ of order $m$ and defect $d$ we can construct a new Langford sequence $L'$ of order $mn$ and defect $nd-(n-1)/2$. We believe that the proof that we provide is remarcable since it is a constructive technique to obtain such sequences. It is precisely this constructive technique that allows us to obtain lower bounds for the number of distinct Langford sequences for certain orders and defects (see Theorem \ref{theo: lower bound c'-Langford from c-Langford} and Corollary \ref{coro: bound_c_langford_from_skolem}).

\begin{theorem}\label{theo: c'-Langford from a previuos c-Langford}
Let $L=(l_1,l_2,\ldots, l_{2m})$ be a Langford sequence of order $m$ and defect $d$. Then, there exists a Langford sequence $L'$ of order $mn$ and defect $d'$, where $d'=nd-(n-1)/2$, for each odd integer $n$.
\end{theorem}
\begin{proof}
Let $L=(l_1,l_2,\ldots, l_{2m})$ be a Langford sequence of order $m$ and defect $d$ and consider the Langford labeling of $D=m\overrightarrow{K_2}$ induced by $L$, where the vertices of $m\overrightarrow{K_2}$ are identified by the labels. Clearly, $\mathcal{RS}_n\subset \Sigma_n$. By Theorem \ref{theo_copies_of_forest}, if we consider any function $h:E(D)\rightarrow \mathcal{RS}_n$, then $D'=D\otimes_h\mathcal{RS}_n\cong (mn)\overrightarrow{K_2}$. Moreover, since a Langford labeling of $m\overrightarrow{K_2}$ is an optimal $1$-equitable labeling of $D$, by Theorem \ref{theo: product_k_equitable} and Remark \ref{remark: induced labeling product}, the induced labeling of $(mn)\overrightarrow{K_2}$, namely $g$, is optimal $1$-equitable. What remains to be proven is that this labeling $g$ is in fact a Langford labeling of $(mn)\overrightarrow{K_2}$ with induced differences in $[d',d'+mn-1]$. That is, the set of induced edge differences is $\{|g(y)-g(x)|: \ (x,y)\in D'\}=[d',d'+mn-1],$ where $d'=nd-(n-1)/2$. Suppose that $(u,v)$ is an arc in $D$. Then, for every $(i,j)\in h((u,v))$, $((u,i),(v,j))\in E(D')$. By Lemma \ref{coro: the unicity of $i-j=k$}, for every integer $k$ with $|k|\le (n-1)/2$ there exists a unique arc $(i,j)\in h((u,v))$ such that $i-j=k$. Thus, for every $(u,v)\in E(D)$, we have that the set $S_{(u,v)}$ of induced edge differences of the arcs obtained from $(u,v)$ is the following one:
\begin{eqnarray*}
  S_{(u,v)}&=&\{|g((u,i))-g((v,j))|: (i,j)\in E(h((u,v))\} \\
  &=&  \{|n(u-1)+i-n(v-1)-j|: (i,j)\in E(h((u,v))\} \\
  &=&[n(v-u)-(n-1)/2, n(v-u)+(n-1)/2].
\end{eqnarray*}
Hence, the set $S=\cup_{(u,v)\in E(D)}S_{(u,v)}$ is $S=[nd-(n-1)/2,nd-(n-1)/2+mn-1]$. Therefore, we obtain a  Langford labeling of $D'$, from which we can recover a Langford sequence of order $mn$ and defect $d'$.\qed
\end{proof}

\begin{example}
Consider Example \ref{Example_1: product of a matching by rotated}. The digraph $D$ is in fact, a Skolem labeling of $4\overrightarrow{K_2}$, and $\Gamma$ is the family $\mathcal{RS}_n$, for $n=3$. Thus, by replacing each vertex $(a,x)$ of $D\otimes_h\Gamma$ by $n(a-1)+x$, we obtain a Langford labeling of order $12$ and defect $d'=2$. This labeling is shown in Fig. \ref{Fig_3}
\begin{figure}[h]
  \centering
  \includegraphics[width=337pt]{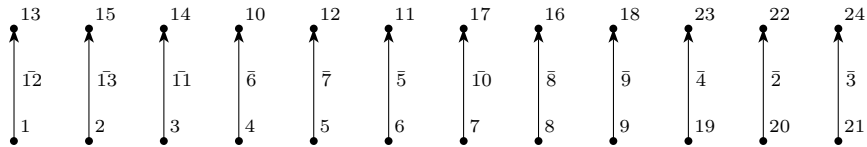}\\
  \caption{A Langford labeling of $12\protect\overrightarrow{K_2}$.}\label{Fig_3}
\end{figure}

Thus, the associated Langford sequence of order $12$ and defect $2$ is:
$$(12,13,11,6,7,5,10,8,9,6,5,7,12,11,13,8,10,9,4,2,3,2,4,3).$$
\end{example}

Denote by $\lambda_n^d$ the number of Langford sequences of order $n$ and defect $d$.

\begin{theorem}\label{theo: lower bound c'-Langford from c-Langford}
Let $m$ and $n$ be two positive integers, $n$ odd. Then,
$$\lambda_{mn}^{d'}\ge |\mathcal{S}_n|^m\lambda_{m}^{d},$$
where $d'=nd-(n-1)/2$.
\end{theorem}
\begin{proof}
By definition, it is clear that $| \mathcal{RS}_n|=| \mathcal{S}_n|$.
Thus, according to the construction shown in the proof of Theorem \ref{theo: c'-Langford from a previuos c-Langford}, it is enough to show that (i) every pair $L_1$ and $L_2$ of different Langford sequences of order $m$ and defect $d$ will produce a pair $L_1'$ and $L_2'$ of different Langford sequences of order $mn$ and defect $d'$, where $d'=nd-(n-1)/2$, and (ii), for a fix Langford sequence $L$ of defect $d$, if we consider two different functions $h_1,h_2:E(D)\rightarrow \mathcal{RS}_n$, where $D$ is the labeled digraph associated to $L$, of size $m$, then the resulting Langford sequences $L_1'$ and $L_2'$ of defect $d'$, are different.
Let us prove (i). Suppose that $D_1$ and $D_2$ are the digraphs associated to $L_1=(l_1^1,l_2^1,\ldots,l_{2m}^1)$ and $L_2=(l_1^2,l_2^2,\ldots,l_{2m}^2)$, respectively. Let $t$ be the minimum $r\in [1,2m]$ such that $l_r^1\ne l_r^2$. Assume that $l_t^1=k$. Then, $k\notin \{l_1^1,l_2^1,\ldots,l_{t-1}^1\}$. Otherwise, we get $l_t^1= l_t^2=k$, that is a contradiction. Thus, there exists
an arc $(t,v_1)\in D_1$ such that $v_1-t=k$, and $(t,v_1)\notin D_2$. Hence, there exists $v_2\in [1,2m]$ such that $(t,v_2)\in E(D_2)$ and $v_2\neq v_1$. Then, for every $i\in [1,n]$, $n(t-1)+i$ is adjacent to exactly one vertex $\{n(v_1-1)+j: \ j\in [1,n]\}$ in $D_1'$ and to exactly one vertex $\{n(v_2-1)+j: \ j\in [1,n]\}$ in $D_2'$. Therefore, the induced $L_1'$ and $L_2'$ Langford sequences of defect $d'$ are different.
Let us prove (ii). Let $L$ be a Langford sequence of defect $d$ and consider two different functions $h_1,h_2:E(D)\rightarrow \mathcal{RS}_n$, where $D$ is the labeled digraph associated to $L$. Then, there exists $(u,v)$ in $D$, such that $h_1((u,v))\neq h_2((u,v))$. Thus, there exists $(i,j)\in E(h_1((u,v)))\setminus E(h_2((u,v)))$. Hence, $n(u-1)+i$ is adjacent to $n(v-1)+j$ in $D_1'$, and is not in $D_2'$. Therefore, the associated Langford sequences of defect $d'$ are different and the result follows.\qed
\end{proof}

In particular, if we use the lower bound given in Theorem \ref{theo_lower_bound_Skolem}, we obtain the next result.
\begin{corollary}\label{coro: bound_c_langford_from_skolem}
Let $m$ and $n$ be two positive integers such that $m\equiv 0$ or $1$ (mod $4$) with $n$ odd. Then,
$$\lambda_{mn}^{(n+1)/2}\ge |\mathcal{S}_n|^m2^{\lfloor m/3\rfloor}.$$
\end{corollary}
\begin{proof}
A Langford sequence of defect $d=1$ is a Skolem sequence. Thus, the result follows from Theorem \ref{theo: lower bound c'-Langford from c-Langford} and Theorem \ref{theo_lower_bound_Skolem}.\qed
\end{proof}
\begin{remark}
Let $m$ and $n$ be two positive integers such that $m\equiv 0$ or $1$ (mod $4$) and let $n$ be odd. Since every cycle admits two possible strong orientations, by Theorem \ref{theo:numberSEMofCICLES}, the bound presented in Corollary \ref{coro: bound_c_langford_from_skolem} implies that
\begin{eqnarray*}
 \lambda_{mn}^{(n+1)/2} &\ge& (\frac52 \cdot 2^{\lfloor (n-1)/3\rfloor}+2)^m2^{\lfloor m/3\rfloor}.
\end{eqnarray*}
\end{remark}


\section{Hooked and extended Skolem sequences}
\label{section_Hooked_Skolem_Langford}
In this section, we consider the problem of generating extended Skolem sequences using hooked and extended Skolem sequences
together with the ideas already developed in the paper.

The fact that there is a position occupied by zero, namely $i$, in the hooked (or extended) Skolem sequences can be
interpreted as position $i$ being occupied by two zeros and hence, we have $|i-i|=0$. In fact, we believe that this is a natural way to understand this type of sequences. Nevertheless, this idea presents some technical difficulties when we pretend to apply the ideas developed so far. For instance, the sequence is no longer related with a $1$-equitable labeling of the digraph $m\overrightarrow{K_2}$ but to a $1$-equitable labeling of the digraph $m\overrightarrow{K_2}\cup \overrightarrow{L},$ where $\overrightarrow{L}$ is a loop. Observe that, in this case, the $1$-equitable labeling is not optimal.
Next, we present what has been said up to this point in a formal way.

For any, hooked or extended Skolem sequence of order $m$, we define an extended Skolem labeling $g$ of the digraph $m\overrightarrow{K_2}\cup \overrightarrow{L}$, up to isomorphism, as follows. Let $f: E(m\overrightarrow{K_2}\cup \overrightarrow{L})\rightarrow [0,m]$ be a bijective function that assigns $0$ to the loop. If $f(u,v)=k$, where $k\neq 0$, then the extended Skolem labeling $g$ assigns to $u$ the smallest position occupied by $k$ and to $v$ the other one. Furthermore, $g$ assigns to the loop vertex the position occupied by zero. See an example in Fig. \ref{Fig_4}.


\begin{figure}[ht]
  \centering
  \includegraphics[width=93pt]{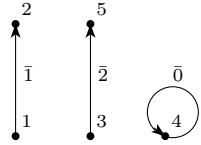}\\
  \caption{The Skolem labeling of $2\protect\overrightarrow{K_2}\cup \protect\overrightarrow{L}$ associated to  $(1,1,2,0,2)$.}\label{Fig_4}
\end{figure}


The following lemma will be useful.
\begin{lemma}\label{lemma: Rotated_canonic}
Let $n$ be odd. Then the digraph formed by a loop and the union of cyclic oriented digons is a rotation super edge-magic digraph of order $n$.
\end{lemma}

\begin{proof}
Let $C_n$ be the graph with vertex set $V(C_n)=\{v_i\}_{i=1}^n$ and set of edges $E(C_n)=\{v_iv_{i+1}\}_{i=1}^{n-1}\cup \{v_1v_n\}.$ Consider the super edge-magic labeling $f$ of $C_n$, $n$ odd, (see \cite{KotRos70}) defined by:
$$
   f(v_i)=\left\{\begin{array}{ll}
                   (i+1)/2, & i \ \hbox{odd}, \\
                   (n+i+1)/2, & i \ \hbox{even}.
                 \end{array}\right.
$$
Let $S_1$ and $S_2$ be the two possible strong orientations obtained from the induced the super edge-magic labeled graph. Then, the two rotations $R_1$ and $R_2$ meet the required conditions.\qed
\end{proof}

\begin{figure}[h]
  \centering
  \includegraphics[width=178pt]{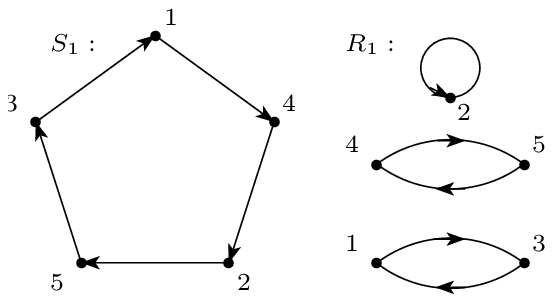} \includegraphics[width=178pt]{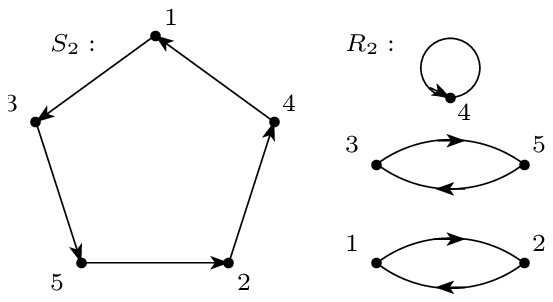}\\
  \caption{The digraphs $S_1$, $S_2$, $R_1$ and $R_2$ introduced in Lemma \ref{lemma: Rotated_canonic}, for $n=5$.}\label{Fig_5}
\end{figure}

\begin{theorem}\label{theo: Extended-Skolem from a previuos Extended-Skolem}
Let $(s_1,s_2,\ldots, s_{2m+1})$ be a hooked or extended Skolem sequence of order $m$. Then, there exists an extended sequence $L'$ of order $mn+(n-1)/2$, for each odd integer $n$.
\end{theorem}
\begin{proof}
Let $(s_1,s_2,\ldots, s_{2m+1})$ be either a hooked or an extended Skolem sequence of order $m$
 and consider the digraph $D=m\overrightarrow{K_2}\cup \overrightarrow{L}$ induced by the sequence, where the vertices of $m\overrightarrow{K_2}\cup \overrightarrow{L}$ are identified with the corresponding labels. By definition of the $\otimes_h$-product, for any function $h:E(D)\rightarrow \mathcal{RS}_n$, we have that $D\otimes_h\mathcal{RS}_n\cong (m\overrightarrow{K_2}\otimes_{h_1}\mathcal{RS}_n)\cup R$, where $h_1$ is the restriction of $h$ over $E(m\overrightarrow{K_2})$ and $R$ is the digraph of $\mathcal{RS}_n$ assigned to the loop $\overrightarrow{L}$. Assume that $R$ is either $R_1$ or $R_2$ introduced in the proof of Lemma \ref{lemma: Rotated_canonic}.
Clearly, $\mathcal{RS}_n\subset \Sigma_n$. By Theorem \ref{theo_copies_of_forest}, if we consider any function $h_1:E(m\overrightarrow{K_2})\rightarrow \mathcal{RS}_n$, then $m\overrightarrow{K_2}\otimes_h\mathcal{RS}_n\cong (mn)\overrightarrow{K_2}$.

Moreover, since an extended Skolem labeling of $m\overrightarrow{K_2}\cup \overrightarrow{L}$ induces a complete $1$-equitable labeling of $m\overrightarrow{K_2}$ with set of induced differences $[1,m]$, by Theorem \ref{theo: product_k_equitable} and Remark \ref{remark: induced labeling product}, the induced labeling of $(mn)\overrightarrow{K_2}$ is $1$-equitable. What remains to be proven is that from this labeling we can obtain an extended labeling $g$ of $(mn+(n-1)/2)\overrightarrow{K_2}\cup \overrightarrow{L}$ with induced differences in $[0,mn+(n-1)/2]$. That is, the set of induced edge differences is $\{|g(y)-g(x)|: \ (x,y)\in D'\}=[0,mn+(n-1)/2]].$  Suppose that $(u,v)$ is an arc in $D$. Then, for every $(i,j)\in h((u,v))$, $((u,i),(v,j))\in E(D')$. By Lemma \ref{coro: the unicity of $i-j=k$}, for every integer $k$ with $|k|\le (n-1)/2$ there exists a unique arc $(i,j)\in h((u,v))$ such that $i-j=k$. Thus, for every $(u,v)\in E(D)$, we have that the set $S_{(u,v)}$ of induced edge differences of the arcs obtained from $(u,v)$ is the following one:
\begin{eqnarray*}
  S_{(u,v)}&=&\{|g((u,i))-g((v,j))|: (i,j)\in E(h((u,v))\} \\
  &=&  \{|n(u-1)+i-n(v-1)-j|: (i,j)\in E(h((u,v))\} \\
  &=&[n(v-u)-(n-1)/2, n(v-u)+(n-1)/2].
\end{eqnarray*}
Hence, the set $S=\cup_{(u,v)\in E(m\overrightarrow{K_2})\cup \overrightarrow{L}}S_{(u,v)}$ is $S=[-(n-1)/2,mn+(n-1)/2]$. Notice that, by Lemma \ref{coro: the unicity of $i-j=k$} the set of induced differences of the arcs in $R$ is the set $[-(n-1)/2,(n-1)/2]$. Hence, removing the negative arc from every digon, we obtain the digraph $R'\cong ((n-1)/2)\overrightarrow{K_2}\cup \overrightarrow{L}$, with set of induced differences $[0,(n-1)/2]$. Therefore, the resulting labeling of $(mn+(n-1)/2)\overrightarrow{K_2}\cup \overrightarrow{L}$ is an extended Skolem labeling from which we can recover an extended Skolem sequence of order $mn+(n-1)/2$.\qed
\end{proof}

\begin{example}
Let $S_3$ be the super edge-magic labeled digraph of order $5$ defined by $V(S_3)=[1,5]$ and $E(S_3)=\{ (1,5),(5,2),(2,3), (3,1),(4,4)\}$. By rotating its adjacency matrix $\pi/2$ radiants clockwise, we obtain the digraph with $V(R_3)=[1,5]$ and $E(R_3)=\{(1,3), (3,4),(4,2),(2,1),(5,5)\}$. Consider $\Gamma=\{RS_i\}_{i=1}^3$, where  $R_1$ and $R_2$ are the digraphs that appear in Fig. \ref{Fig_5}. Let $D$ the digraph that appears in Fig. \ref{Fig_4} and let $h: E(D)\rightarrow \Gamma$ be the function defined by $h((1,2))=R_3$, $h((3,5))=R_2$ and $h((4,4))=R_1$. Then, by replacing each vertex $(a,x)$ of $D\otimes_h\Gamma$ by $5(a-1)+x$, we obtain the extended Skolem labeling of $12\overrightarrow{K_2}\cup \overrightarrow{L}$ of Fig. \ref{Fig_6}. The induced extended Skolem sequence is:
\begin{equation*}
    (7,4,6,3,5,4,3,7,6,5,11,9,12,10,8,2,0,2,1,1,9,11,8,10,12).
\end{equation*}

\begin{figure}[h]
  \centering
  \includegraphics[width=357pt]{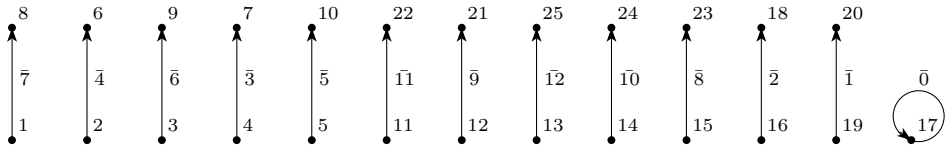}
  \caption{An extended Skolem labeling of $12\protect\overrightarrow{K_2}\cup\protect\overrightarrow{ L}$.}\label{Fig_6}
\end{figure}

\end{example}

Denote by $\epsilon_m$ the number of extended Skolem sequences of order $m$.

\begin{theorem}\label{theo: lower bound extended Skolem from extended_Skolem}
Let $m$ and $n$ be two positive integers, $n$ odd. Then,
$$\epsilon_{mn}\ge 2 |\mathcal{S}_n|^m\epsilon_m.$$
\end{theorem}
\begin{proof}
The proof is similar to the one of Theorem \ref{theo: lower bound c'-Langford from c-Langford}. The only difference is that, when we consider the function $h$, the loop only has two possible images under $h$, namely the digraphs $R_1$ or $R_2$ introduced in the proof of Lemma \ref{lemma: Rotated_canonic}.\qed
\end{proof}

\section{Conclusions}
The goal of this paper is to show a new application of labeled super edge-magic digraphs to a well known and deeply studied problem: Skolem and Langford sequences.
It is possible to find in the literature applications of Skolem and Langford sequences to graph labelings. However, we are not aware of applications of graph labelings to construct Skolem and Langford sequences. In this paper, we use super edge-magic labelings of digraphs in order to get an exponential number of Langford sequences with certain orders and defects. Furthermore, we also obtain, using similar techniques, an exponential number of extended hooked sequences.
Recall that, for every Skolem sequence, we can associate a trivial extended Skolem sequence just by placing a $0$ in the last position of the sequence. This fact is not very interesting by itself. Nevertheless, these type of extended Skolem sequences may be interesting, since using them together with (certain) super edge-magic labelings will allow us to produce many extended (hooked) Skolem sequences. We have introduced this approach to the problem of finding Langford, extended and hooked Skolem sequences since we believe that it is a new and unexpected way to attack the problem. This introduces new light into an old problem that has proven to be very hard.

\noindent {\bf Acknowledgements}
The research conducted in this document by the first author has been supported by the Spanish Research Council under project
MTM2011-28800-C02-01 and symbolically by the Catalan Research Council
under grant 2014SGR1147.



\end{document}